\documentclass[a4paper,10pt]{amsart}
\usepackage[latin1]{inputenc}
\usepackage[T1]{fontenc}
\usepackage{bbm}
\usepackage{amsmath}
\usepackage{amsthm}
\usepackage{amssymb}
\usepackage{color}
\usepackage{latexsym}
\usepackage{ulem}
\usepackage[all]{xy}
\usepackage{calc}
\usepackage{soul}
\usepackage{multicol}
\newtheorem{thm}{Theorem}[section]
\newtheorem{prop}{Proposition}[section]

\newtheorem{cor}{Corollary}[section]

\newtheorem{rmq}{Remark}[section]

\newcommand{\R}{\mathbb{R}}
\numberwithin{equation}{section}
\providecommand{\sm}{\setminus}
\providecommand{\R}{\mathbb{R}}

\providecommand{\skp}[2]{\langle#1,#2\rangle}  
\providecommand{\bigskp}[2]{\Big\langle#1,#2\Big\rangle}
\newcommand{\N}{\mathbb{N}}
\newcommand{\Z}{\mathbb{Z}}
\newcommand{\Hy}{\mathbb{H}}

\DeclareMathOperator{\Real}{Re}

\DeclareMathOperator{\supp}{supp}

\newcommand{\stkout}[1]{\ifmmode\text{\sout{\ensuremath{#1}}}\else\sout{#1}\fi}
\newcounter{exercice}

\begin{document}

\title[Helmholtz equations in the hyperbolic space]{On Helmholtz equations and counterexamples to Strichartz
estimates in hyperbolic space}

\author{Jean-Baptiste Casteras}
\address{Jean-Baptiste Casteras
\newline \indent D\'epartement de Math\'ematiques, Universit\'e Libre de Bruxelles,
\newline \indent CP 214, Boulevard du triomphe, B-1050 Bruxelles, Belgium,
\newline \indent and INRIA- team MEPHYSTO.}
\email{jeanbaptiste.casteras@gmail.com}
\author{Rainer Mandel}
\address{Rainer Mandel
\newline \indent Karlsruhe Institute of Technology
\newline \indent  Institute for Analysis
\newline \indent  Englerstrasse 2, D-76131 Karlsruhe, Germany.}
\email{Rainer.Mandel@kit.edu }

\maketitle

\allowdisplaybreaks

\section*{Abstract}

In this paper, we study nonlinear Helmholtz equations 
 \begin{equation} \tag{NLH}
\label{NLHabs}
   -\Delta_{\Hy^N} u - \frac{(N-1)^2}{4} u -\lambda^2 u = \Gamma|u|^{p-2}u  
   \quad\text{in }\Hy^N, \;N\geq 2
\end{equation}
 where $\Delta_{\Hy^N}$ denotes the Laplace-Beltrami operator in the hyperbolic space $\Hy^N$ and
 $\Gamma\in L^\infty(\Hy^N)$ is chosen suitably. Using fixed point and variational
 techniques, we find nontrivial solutions to~\eqref{NLHabs} for all $\lambda>0$ and $p>2$.
 The oscillatory behaviour and decay rates of radial solutions is analyzed, 
 with possible extensions to Cartan-Hadamard manifolds and Damek-Ricci spaces.  
 Our results rely on a new Limiting Absorption Principle for the Helmholtz operator in $\Hy^N$. As a
 byproduct, we obtain simple counterexamples to certain Strichartz estimates.

\section{Introduction}

 In this paper we are interested in nontrivial solutions of the Nonlinear Helmholtz Equation (NLH)  
 \begin{equation} \label{eq:NLH}
   -\Delta_{\Hy^N} u - \frac{(N-1)^2}{4} u -\lambda^2 u = \Gamma|u|^{p-2}u  
   \quad\text{in }\Hy^N
\end{equation}
 where $\Delta_{\Hy^N}$ denotes the Laplace-Beltrami operator in hyperbolic space $\Hy^N,N\geq 2$ and
 $\Gamma\in L^\infty(\Hy^N)$. As in the Euclidean setting, linear and nonlinear Helmholtz equations
 arise from a standing wave ansatz for the corresponding Schr\"odinger or wave equations that have attracted much
 interest in the last
 years~\cite{Anker,AnkPi_WaveKG,AnPiVa_WaveEquation,BanCarSta_Scattering,BaDu1,BaDu,BGH,Io,MetTay_NLWaves,Tat_Strichartz},
 especially concerning Strichartz estimates. In order to motivate our first result on the failure of
 Strichartz estimates in hyperbolic space and to provide the link to Helmholtz equations, let us first 
 review the situation in $\R^N$.
 
 \medskip
 
 For the homogeneous Schr\"odinger equation, it was Strichartz himself who proved the (global) 
 Strichartz estimate
 \begin{equation}\label{eq:Strichartz}
   \begin{cases}
     ~\; i\partial_t\psi - \Delta \psi = 0 \quad \text{in }\R^N,\qquad \psi(0)=\psi_0 \\
     ~\; \|\psi\|_{L^p(\R;L^q(\R^N))} \leq C\|\psi_0\|_{L^r(\R^N)},
   \end{cases}
 \end{equation}
 for $r=2, p=q= 2(N+2)/N$, see Corollary~1~\cite{Stri_Restrictions}. The proof is based on Fourier
 restriction theory for paraboloids, which in turn relies on the Stein-Tomas theorem~\cite{Tom_A_restriction}.
 Since then, many generalizations of such estimates to more general $r,p,q$ and other dispersive PDEs have
 been found. The topic being quite vast and intensively studied until today, we do not make any attempt to
 present a comprehensive list of related results. For a detailed treatment of the Schr\"odinger equation, we refer
 to Cazenave's book~\cite{Caz_SLSchroedinger}.  Let us only mention that the scaling invariance 
 of the Schr\" odinger equation shows that in $\R^N$ the estimate~\eqref{eq:Strichartz} can only hold if
 $2/p+N/q=N/r$. 
 Homogeneous Strichartz estimates~\eqref{eq:Strichartz} are known to hold for certain ranges of exponents
 $p,q$ with $r\in (1,2]$, but, up to the authors' knowledge, nothing is known for $r>2$. For
 $r>2N/(N-1)$, it follows from the theory of Helmholtz equations that no dispersive estimate and
 especially none of the above estimates (except for $p=\infty$) can hold.
 Indeed, the method of stationary phase shows that certain solutions to the Helmholtz
 equation $-\Delta \psi_0-\omega\psi_0=0$ for $\omega>0$, namely Herglotz waves given by a sufficiently smooth density
 over the sphere, decays exactly like $|x|^{(1-N)/2}$ as $|x|\to\infty$ (Theorem~1a \cite{Kato_growth},
 Proposition~1 \cite{Man_uncountably}).
 In particular, $\psi(x,t):=e^{i\omega t}\psi_0(x)$ is a solution of the NLS with
 initial datum in $L^r(\R^N)$ precisely for $r>2N/(N-1)$ that does not disperse as $t\to\infty$. We believe
 that it is an interesting open question, whether or not Strichartz estimates hold for initial data $\psi_0\in
 L^r(\R^N)$ with $2<r\leq 2N/(N-1)$.
 
 \medskip
 
 In hyperbolic space, homogeneous Strichartz estimates of the form
 \begin{equation}\label{eq:Strichartzhyp}
   \begin{cases}
     ~\; i\partial_t\psi - \Delta_{\Hy^N} \psi = 0 \quad \text{in }\Hy^N,\qquad \psi(0)=\psi_0 \\
     ~\; \|\psi\|_{L^p(\R;L^q(\Hy^N))} \leq C\|\psi_0\|_{L^r(\Hy^N)}
     \end{cases}
 \end{equation}
 hold for $r=2$ and $p\in [2,\infty),q\in [2,\infty]$ with
 $\frac{2}{p}+\frac{N}{q}\geq \frac{N}{r}$, see Theorem~3.6~\cite{Anker}. In particular, restricting the
 attention to the classical case $r=2$, one sees that Strichartz estimates hold for more exponents than in the
 Euclidean setting. 
 Again, nothing seems to be known for $r>2$.  Given the above considerations in $\R^N$, one
 way of disproving the validity of Strichartz estimates is to determine the decay rate of solutions of
 the linear homogeneous Helmholtz equation in $\Hy^N$. Since we will prove that these solutions lie in
 $L^r(\Hy^N)$ for all $r\in (2,\infty]$, we infer the following.  
 
 \begin{thm}\label{thm:NoStrichartz}
   Let $p\in [1,\infty),q,r\in [1,\infty]$. Then the homogeneous Strichartz estimate for the
   Schr\"odinger equation in $\Hy^N$~\eqref{eq:Strichartzhyp} can only hold provided $1\leq r\leq 2$.
 \end{thm}

 \begin{rmq}
   The analogous statement holds for the initial value problem for the wave equation in $\Hy^N$ and thereby
   complements the results on Strichartz estimates 
   from~\cite{AnkPi_WaveKG,MetTay_NLWaves,Pierf_WeightedStrichartz,Tat_Strichartz}.
 \end{rmq}
 
  Next we present our results for the Nonlinear Helmholtz Equation~\eqref{eq:NLH}, which has not  been
  considered in the literature so far. For $-\lambda^2$ replaced by
  $+\lambda^2$, results on positive and sign-changing solutions can for instance be found
  in~\cite{GanSan_Signchanging,GanSan_nondegenacy,ManSan_OnASemilinear}. We stress that the techniques used
  in these papers are in spirit close to their Euclidean counterparts and the latter change drastically
  according to the sign in front of $\lambda^2$. A much more helpful reference are the
  papers by Guti\'{e}rrez \cite{Gu}, Evequoz, Weth \cite{Ev,EW} and~\cite{Man_uncountably,MMP} where the
  Nonlinear Helmholtz Equation was studied in Euclidean space. In this setting the operator $
  -\Delta_{\Hy^N} - (N-1)^2/4$ is replaced by the negative Euclidean Laplacian so that 0 is again the bottom
  of the essential spectrum.
  Our intention is to demonstrate that \eqref{eq:NLH} can be handled
  much more easily compared to its Euclidean analogue, which is due to a stronger Limiting Absorption
  Principle for the Helmholtz operator $L-\lambda^2$ that we will prove in Section~\ref{sec:ResEst}. Using these results we can follow the
  lines of~\cite{Gu,Man_uncountably} and prove the existence of uncountably many small solutions via a fixed
  point argument.

  \begin{thm}\label{thm:existence_small_solutions}
    Let $N\geq 2, \Gamma\in L^\infty(\Hy^N),\lambda>0$ and $p>2$. Then~\eqref{eq:NLH} has uncountably many
    small solutions lying in $W^{2,r}(\Hy^N)$ for all $r\in (2,\infty)$.
  \end{thm}
 
  This result actually holds under much weaker assumptions on the nonlinearity. 
  In fact, in the proof we will replace $\Gamma |u|^{p-2}u$ by any function $f\in C^1(\R_+\times\R)$
  satisfying $|f(x,z)|+|z||f_z(x,z)|\leq C|z|^{q-1}$ for some $C>0$ and all $x\in\Hy^N,|z|\leq 1$. We
  mention that no upper bound for $p$ is necessary. The solutions from the previous theorem are parametrized by hyperbolic Herglotz waves and
  therefore can even be shown to form a continuum in the above-mentioned Sobolev spaces as
  in~\cite{Man_uncountably}. Large solutions of~\eqref{eq:NLH} can be constructed using a dual variational
  approach as in~\cite{Ev:OnThe,EW}. Note that classical variational approaches are not suitable since
  solutions are not expected to lie in $L^2(\Hy^N)$, as we will show further below. The dual variational
  method yields the following.

\begin{thm}\label{thm:existence_largesolutions}
  Let $N\geq 2,\lambda>0,\Gamma\in L^\infty(\Hy^N)$ satisfy $\Gamma\geq 0,\Gamma\not\equiv 0$ and $2<p<2^*$.
  Then~\eqref{eq:NLH} has a nontrivial solution $u\in W^{2,r}(\Hy^N)$ for all $r\in (2,\infty)$ provided
  \begin{itemize}
    \item[(i)] $\Gamma(x) \geq \Gamma_0$ where $\Gamma_0 = \lim_{d(0,x)\to\infty} \Gamma(x)$.
  \end{itemize} 
  Under the assumptions
  \begin{itemize}
    \item[(ii)] $\Gamma(x)\to 0$ as $d(0,x)\to \infty$ or
    \item[(iii)] $\Gamma$ is radially symmetric about some point in $\Hy^N$
  \end{itemize}
  there is a sequence of nontrivial solutions $u_n\in W^{2,r}(\Hy^N)$ for all $r\in (2,\infty)$ that is
  unbounded in $L^p(\Hy^N)$. In the case (iii), the solutions are radial.
\end{thm}
 
 Here, $0$ stands for the origin in hyperbolic space and $d(x,y)$ denotes the geodesic distance of two points
 $x,y\in \Hy^N$. Radially symmetric functions only depend on the geodesic distance to one particular point in
 $\Hy^N$, which we denote by $r$ in the following. Let us mention that the dual variational method is flexible
 enough to treat also higher order problems as in~\cite{BCM_FourthOrder} or negative $\Gamma$ as in~\cite{MMP}.
 
 \medskip
  
 In our final result, we restrict
 our attention to radially symmetric solutions. One motivation is our interest in exact pointwise decay
 properties of solutions to~\eqref{eq:NLH}, which do not seem to be availabe in general.
  Let us point out that, using the ideas of Lemma $2.9$ in~\cite{EW}, it is possible to prove that
  the solutions constructed in~\eqref{thm:existence_largesolutions} decay at least like $e^{(1-N)r/2}$ as $r\to\infty$ provided that
  $p>4$. Not being convinced in the optimality of this result, we dispense with a proof. 
  In the radial case we can show with elementary means that the solutions decay exactly like
  $e^{(1-N)r/2}$ as $r\to\infty$ and, in particular, do not lie in $H^1(\Hy^N)$. We show this to hold without
  any restriction on $p$. As a consequence those solutions can not be found with classical variational
  methods, as we mentioned above. 
  Let us point out that finite energy solutions of nonresonant problems in $\Hy^N$ as
  in~\cite{ManSan_OnASemilinear} decay faster, see Remark~3.8 in that paper. Moreover, we prove that radial
  solutions oscillate. 
 Another interesting feature is that we can analyze radial solutions on manifolds $M$ which are much more
 general than $\Hy^N$ such as Damek-Ricci spaces \cite{ADY,APV15,Pierf_WeightedStrichartz}. Indeed, it is known
 that the radial part of the Laplace-Beltrami operator on such a manifold is given by 
 $$
    \partial_{rr} + \frac{f'(r)}{f(r)}\partial_r
    \qquad\text{where }
    f(r)= \sinh^{m+k}(\frac{r}{2}) \cosh^k(\frac{r}{2}),
 $$ 
 see (2.11)~\cite{ADY}. Here, $m,k\in\N$ and the dimension of the manifold is $m+k+1$. The corresponding
 formula also holds in hyperbolic space ($f(r)=\sinh(r)^{N-1}$) and Euclidean space ($f(r)=r^{N-1}$) and even
 more general classes of rotationally symmetric Cartan-Hadamard manifolds. Each of these examples satisfies
 assumption (H1) that we will need.
 The full set of conditions reads as follows:   
\begin{itemize}
  \item[(H1)] $f\in C^1(\R)$ with $f'>0$ and such that $\log(f)'(r)\to \kappa$, $\log(f)''(r)\to 0$ as
  $r\to\infty$ and $(\log(f)')^2-\kappa^2$ is integrable near infinity 
  for some $\kappa\in [0,\infty)$. 
  \item[(H2)] $V\in C^1(\R),V>0$ with $V(r)\to V_\infty>\kappa^2/4$ and $V'\in L^1(\R_+)$
  \item[(H3)] $\Gamma\in C^1(\R),\Gamma\geq 0$ with $\Gamma(r)\to \Gamma_\infty\geq 0$
  and $|\Gamma'|/\Gamma\in L^1(\R_+)$ or $\Gamma\equiv 0$.
\end{itemize}
 Under these conditions we prove the following:

\begin{thm}\label{thm:radial}
  Assume (H1),(H2),(H3) and $p>2$. Then the solution $u_\gamma$ of 
  $$
    - u''(r) - \frac{f'(r)}{f(r)} u(r) - V(r)u = \Gamma(r)|u|^{p-2}u \quad\text{on }[0,\infty),\quad
    u(0)=\gamma,\;u'(0)=0
  $$  
  has infinitely many zeros and satisfies for all $r\geq 0$ 
  \begin{align} \label{notL2}
    \begin{aligned}
      |u_\gamma (r)|^2+|u^\prime_\gamma (r)|^2   
      &\leq  C(V(0)\gamma^2+\Gamma(0)|\gamma|^p)(1+f(r)), \\
     |u_\gamma (r)|^2+|u^\prime_\gamma (r)|^2  
      &\geq  c(V(0)\gamma^2+\Gamma(0)|\gamma|^p)(1+f(r)) 
    \end{aligned}
  \end{align}
  where $c,C>0$ are independent of $\gamma$. 
\end{thm}

 The proof of Theorem~\ref{thm:radial} partly generalizes and improves Theorem~1.2~\cite{MMP} given that we
 can allow for a quite large class of functions $f$ and that the bounds on the right hand side
 in~\eqref{notL2} are more explicit. Moreover, though being similar, the proof is much shorter.
 As in Theorem~1.2  or Theorem~2.10 in~\cite{MMP} the method of
 proof is also suitable for more general nonlinearities, including $\Gamma |u|^{p-2}u$ with
 negative $\Gamma$. Note that in this case one can show that radial solutions are unbounded if $|\gamma|$ is
 large and bounded for small $|\gamma|$ provided some mild additional assumptions on $f,V,\Gamma$ are
 satisfied.
 
 \medskip

 Let us give a short outline of this paper and comment on the notation that we will employ. In
 Section~\ref{sec:ResEst} we will prove resolvent estimates for Helmholtz operators in $\Hy^N$ and use them
 for the proof of a Limiting Absorption Principle. This will be used  to
 prove (very quickly) Theorem~\ref{thm:NoStrichartz} in Section~\ref{sec:Nostrichartz} and 
 Theorem~\ref{thm:existence_small_solutions} via a fixed point argument in Section~\ref{sec:SmallSolutions}. 
 In Section~\ref{sec:LargeSolutions} we implement the dual variational method following \cite{EW} in order to
 prove Theorem~\ref{thm:existence_largesolutions}. In the final section, we prove Theorem~\ref{thm:radial}. 
 In the following, $C$ denotes a generic constant that may change from line to line. The
 $N$-dimensional hyperbolic space $\Hy^N=\{x\in\R^N: x_N>0\}$ is considered in the half space model with
 geodesic distance $d(x,y)=2\arcsin(|x-y|/(2\sqrt{x_Ny_N}))$ and volume element
 $dV= x_N^{-2}((dx')^2+dx_N^2 )= \sinh(r)^{N-1}\,dr\,d\theta$. Its Laplace-Beltrami operator is given by 
 $\Delta_{\Hy^N} := x_N^2(\partial_1^2+\ldots+\partial_N^2) -  (N-2)x_N \partial_N$.


\section{Resolvent estimates} \label{sec:ResEst}

In this section we discuss resolvent estimates for the operators $L-(\lambda+i\mu)^2$ for
$\lambda>0,\mu\neq 0$ and 
$$
  L := -\Delta_{\Hy^N} -\dfrac{(N-1)^2}{4}.
$$
It is well-known that the spectrum of $L$ is given by $\sigma(L)=[0,\infty)$ so that the resolvent of
$L-\lambda^2$ does not exist in the classical sense. However, as in the Euclidean case, it is possible to
prove a Limiting Absorption Principle which yields a solution of the linear Helmholtz equation
$Lu-\lambda^2 u = f$ for functions $f$ that decay sufficiently fast at infinity, see
Theorem~I.4.2~\cite{IsoKur_Introduction}.
In this approach the resolvents of $L-(\lambda+i\mu)^2$ are studied and function spaces are identified,
in which the limits of the resolvent operators persist as their imaginary parts 
tend to zero from the right respectively from the left. These operators will in the following be denoted by
$(L-\lambda^2-i0)^{-1}$ respectively $(L-\lambda^2+i0)^{-1}$. In the Euclidean setting, estimates for  
$(-\Delta-\lambda^2-i0)^{-1}$ from   weighted Lebesgue spaces $L^{2,s}(\R^N)$ to $L^{2,-s}(\R^N)$
($s>\frac{1}{2}$) or even from $B(\R^N)$ to its dual $B^*(\R^N)$ are due to 
Ikebe and Saito \cite{IkebeSaito_LAP} (Theorem~1.2) as well as Agmon and H\" ormander, see Theorem~4.1
in~\cite{Agmon_Spectral}, \cite{AgHoe_Asymptotic} and Theorem~3.1 in~\cite{Agmon}.
Here,
\begin{align*}
  L^{2,s}(\R^N) &= \big\{ v\in L^2_{loc}(\R^N) : |\cdot|^sv \in L^2(\R^N) \big\}, \\
  B(\R^N) &= \Big\{ v\in L^2_{loc}(\R^N) : \sum_{j=1}^\infty 2^{\frac{j-1}{2}} \Big(
  \int_{\{2^{j-1}<|x|<2^j\}} |u(x)|^2 \,dx\Big)^{1/2} < \infty \Big\}.
\end{align*}
The counterpart for the latter estimate in the hyperbolic case was proved by Perry, see
Theorem~5.1 in~\cite{Perry_LaplaceOperator} or Theorem~I.4.2 in~\cite{IsoKur_Introduction}. In
Theorem 1.2 of~\cite{HS} Huang and Sogge proved that these operators may as well be defined as bounded linear operators
from $L^p(\Hy^N)$ to $L^q(\Hy^N)$,  where $p,q$ satisfy 
$$ 
  \frac{1}{p}-\frac{1}{q}=\frac{2}{N} \quad\text{and}\quad \min\big\{
  \big|\frac{1}{p}-\frac{1}{2}\big|,\big|\frac{1}{q}-\frac{1}{2}\big|\big\} > \frac{1}{2N}.
$$
We stress that these restrictions are essentially due to the authors' focus on uniform estimates with
respect to $\lambda$ within the range $|\lambda|^2\geq 1$. 
For our purposes such a uniform
behaviour is not needed, which allows us to modify and adapt some of the ideas from~\cite{HS} in order to
obtain resolvent estimates for larger ranges of exponents. This extension is based on recent results by
Chen and Hassell~\cite{CH}. For $\sigma_p>0$ given by
 \begin{align*}
   \sigma_p &:= \frac{2N}{p}+N-1  &&\text{if }\; 1\leq p\leq \frac{2(N+1)}{N+3},\\
   \sigma_p &:=  \frac{N-1}{p}-\frac{N-1}{2}  &&\text{if }\; \frac{2(N+1)}{N+3}\leq p <2, 
 \end{align*}
 their result about the spectral resolution $\R\ni \lambda\mapsto E_P(\lambda)$ of the selfadjoint operator
 $P:=\sqrt{L}$ reads as follows.
 
\begin{thm}[Theorem $1.6$ \cite{CH}] \label{thm:CH}
  Let $N\geq 2$ and $1\leq p<2$. Then there a $C>0$ such that the following estimate holds for all
  $\lambda>0$:
  \begin{align*}
    \| \dfrac{d}{d\lambda} E_P (\lambda) \|_{L^p(\Hy^N) \rightarrow L^{p^\prime}(\Hy^N)} 
    &\leq C  \lambda^2 &&(0<\lambda\leq 1)  \\
    \| \dfrac{d}{d\lambda} E_P (\lambda) \|_{L^p(\Hy^N) \rightarrow L^{p^\prime}(\Hy^N)} 
    &\leq C  \lambda^{\sigma_p} &&(\lambda\geq 1)
  \end{align*}
\end{thm}
 
 As we will show below, this estimate may be used in order to prove the resolvent estimates along the
 lines of~\cite{HS}. Before going on with this, we recall some useful information about the Green's function 
  associated with the operator $L-(\lambda+i\mu)^2$ given by $(x,y)\mapsto
 G_{\lambda+i\mu}(d(x,y))$. In other words, for all $f\in C_0^\infty(\Hy^N)$, we have 
 $$
   (L-(\lambda+i\mu)^2)^{-1} f 
   = G_{\lambda+i\mu}\ast f
   := \int_{\Hy^N} G_{\lambda+i\mu}(d(x,y))f(y)\,dV(y).
 $$ 
For notational convenience, we will in the following assume $\mu>0$. It is known (\cite{Tay_PDEsII} p.125)
that there are complex constants $c_N\neq 0$ such that, for odd space dimensions $N$, we have
\begin{equation}
\label{greenodd}
G_{\lambda+i\mu}(t)
  =   \dfrac{c_N }{i\lambda -\mu}\left(\dfrac{1}{\sinh t} \dfrac{\partial}{\partial
t}\right)^{\frac{N-1}{2}} \Big[e^{(i\lambda-\mu) t}\Big],
\end{equation}
whereas in the case of even $N$ we have
\begin{equation}
\label{greeneven}
G_{\lambda+i\mu}(t) 
=      \dfrac{c_N }{i\lambda -\mu} \int_t^\infty \dfrac{\sinh s}{\sqrt{\cosh s -
\cosh t}} \left(\dfrac{1}{\sinh s} \dfrac{\partial}{\partial s}\right)^{\frac{N}{2}} \Big[e^{(i\lambda - \mu)s}\Big]\,ds.
\end{equation}
  The properties of the Green's function are summarized in the following proposition.

 \begin{prop}\label{prop:GreenFunctionH}
   Let $N\in\N,N\geq 2$ and let $G_{\lambda+i\mu}$ be given by \eqref{greenodd} for odd $N$ and
   by \eqref{greeneven} for even $N$. Then for all $\Lambda>0$ there is a constant $C>0$ such that  
   \begin{align*} 
	  |G_{\lambda+i\mu}(t)|&\leq C\max\{t^{2-N},|\log(t)|\} &&(|t|\leq 1), \\
	  |G_{\lambda+i\mu}(t)|&\leq Ce^{(\frac{1-N}{2}-|\mu|) t}   &&(|t|\geq 1)
   \end{align*} 
   for all $\lambda\in [\Lambda^{-1},\Lambda]$ and $\mu\in [-\Lambda,\Lambda]\sm\{0\}$. 
 \end{prop}

\medskip

We shall also use the formula from  (4.4) in~\cite{HS} that allows to write the convolution
$G_{\lambda+i\mu}\ast f$ in a different way. It reads
\begin{equation}\label{eq:formula_resolvent}
  (L - (\lambda +i\mu)^2 )^{-1} f
  = -\frac{1}{i\lambda -\mu} \int_0^\infty e^{(i\lambda -\mu) t}\cos(tP) f\,dt
\end{equation} 
where the function $\cos(tP)$ is defined via functional calculus. Note that this formula is a
consequence of the fact that for any given test function $f$ the function $u(t):=\cos(tP)f$ is the
unique solutions of the initial value problem $\partial_{tt} u + P^2 u=0$, $u(0)=f$, $\partial_t u(0)=0$.
So $L=P^2$ yields
\begin{align*}
  &(L - (\lambda +i\mu)^2) \left( -\frac{1}{i\lambda -\mu} \int_0^\infty e^{(i\lambda -\mu) t}\cos(tP)
  f\,dt \right) \\
  &= -\frac{1}{i\lambda -\mu} \int_0^\infty e^{(i\lambda -\mu) t}  (P^2 - (\lambda +i\mu)^2)u(t)\,dt
  \\
  &= -\frac{1}{i\lambda -\mu} \int_0^\infty e^{(i\lambda -\mu) t}  \left(-\partial_{tt} u(t) + (i\lambda
  -\mu)^2 u(t)\right) \,dt  \\
  &= -\frac{1}{i\lambda -\mu}  \left[ e^{(i\lambda -\mu) t} \big( -\partial_t u(t) + (i\lambda  -\mu)
  u(t)\big) \right]_0^\infty \\
  &= u(0) = f.  
\end{align*}
With these preparations, we can now prove the resolvent estimates for $L$.

\begin{thm}\label{thm:Resolvent_estimates}
Let $N\geq 2$ and $\Lambda>0$. Then there exists a constant $C$ such that 
$$
  \| ( L- (\lambda+i\mu)^2)^{-1}f\|_{L^q(\Hy^N)} \leq C \|f\|_{L^p(\Hy^N)}
$$
for all $\lambda\in [\Lambda^{-1},\Lambda]$ and $\mu\in [-\Lambda,\Lambda]\sm\{0\}$ 
 provided that $1\leq p<2<q$ and $\frac{1}{p}-\frac{1}{q}\leq \frac{2}{N}$ with 
 $(p,q)\neq (1,\frac{N}{N-2}),(\frac{N}{2},\infty)$. 
\end{thm}
\begin{proof}
  As above we only discuss the case $\mu>0$. 
  Let $\beta \in C_0^\infty ((1/2 ,2))$  be a nonnegative function  such that  $\sum_{k\in \Z } \beta
  (2^{-k}t)=1$ for all $t>0$ and set
  $$
    \beta_k(t):= \beta(2^{-k}t)\quad (k\geq 1),\qquad 
    \beta_0(t):= 1 -\sum_{k=1}^\infty \beta_k(t).
  $$
  In particular, we have $\beta_0(t)=1$ for $t\in [0,1]$, $\beta_0(t)=0$ for $t\geq 2$ and the
  $\beta_k$ are supported on annuli with inner resp. outer radius $2^{k-1},2^{k+1}$. Recalling
  \eqref{greenodd} and \eqref{greeneven}, we define, for all $k\in\N_0$ and odd space dimensions $N$, the
  function
  \begin{equation*}
  S_k (t) := \dfrac{c_N}{i\lambda -\mu} 
    \left(\dfrac{1}{\sinh t} \dfrac{\partial}{\partial t}\right)^{\frac{N-1}{2}} 
    \Big[ \beta_k (t) e^{(i\lambda - \mu) t }\Big].
\end{equation*}
  For even $N$ the corresponding definition is
\begin{equation*}
S_k (t):= \dfrac{c_N}{i\lambda -\mu} \int_t^\infty \dfrac{\sinh s}{\sqrt{\cosh s - \cosh t}} 
\left(\dfrac{1}{\sinh s} \dfrac{\partial}{\partial s}\right)^{\frac{N}{2}} \Big[\beta_k(s) e^{(i\lambda  -\mu) s }\Big]\,ds.
\end{equation*}
  These definitions and Proposition~\ref{prop:GreenFunctionH} give $\sum_{k=0}^\infty S_k =
  G_{\lambda+i\mu}$ so that we have to estimate the integrals $S_k\ast f$ for $f\in L^p(\Hy^N)$.

  \medskip
  
  We start with the estimates for $S_0\ast f$. By definition of $\beta_0,S_0$ and 
  Proposition~\ref{prop:GreenFunctionH} we have
  $$ 
     S_0(t) =0 \quad\text{for }t\geq 2,\qquad 
     |S_0(t)|\leq C \max\{t^{2-N},|\log(t)|\} \quad\text{for } 0<t\leq 2. 
  $$
  So $S_0 \in L^r(\Hy^N)$ for $1\leq r<\frac{N}{N-2}$ as well as 
  $S_0 \in L^{\frac{N}{N-2},\infty}(\Hy^N)$ if $N\geq 3$. Using the Weak Young Inequality in $\Hy^N$ (see
  Theorem 6.2.3~\cite{Simon_HarmAna} and the following remarks) 
  we obtain  
  $$
    S_0\ast f \in L^q(\Hy^N) \quad\text{if } 1+\frac{1}{q}=\frac{1}{r}+\frac{1}{p},\; 1\leq r<\frac{N}{N-2}
    \text{ or }r=\frac{N}{N-2},1<p,q<\infty. 
  $$
  In other words, we have
  $$
    S_0\ast f \in L^q(\Hy^N) \quad\text{if } 0\leq \frac{1}{p}-\frac{1}{q}< \frac{2}{N}\;\text{ or }\;
    \frac{1}{p}-\frac{1}{q}=\frac{2}{N},\, 1<p,q<\infty.
  $$
  Since these conditions are satisfied by our assumptions on $p,q$, it remains to estimate the integrals
  $S_k\ast f$ for $k\geq 1$.

\medskip

 Next, we are going to show that, for $k\geq 1$,  
\begin{equation*}
 \|S_k \ast f\|_{L^q(\Hy^N ) }\leq C 2^{-k} \|f\|_{L^p(\Hy^N)}.
\end{equation*}
First, we prove the corresponding inequality for $p=1,q=\infty$. In (4.14)~\cite{HS} it is shown
that for any given $M>0$ there is a $C_M>0$ such that
\begin{equation}   \label{HSe2}
 \|S_k\ast f \|_{L^{\infty} (\Hy^N)}
 \leq C_M 2^{-kM}\|f\|_{L^1(\Hy^N)}.
\end{equation}
This is a consequence of the uniform pointwise exponential decay of $G_{\lambda+i\mu}$ at infinity, see Proposition~\ref{prop:GreenFunctionH}.
In order to prove the inequality for all $q>2$ and $p=2$, we make use the formula
$$
  S_k\ast f = \frac{c_N}{i\lambda-\mu} \int_0^\infty \beta_k(t) e^{(i\lambda-\mu) t}
  \cos(tP)f\,dt 
$$
from p.4655~\cite{HS} (which can be proved just as~\eqref{eq:formula_resolvent}). We define
\begin{align} \label{eq:defn_psik} 
  \begin{aligned}
  \psi_k(s)
  &:= \int_\R \beta_k(t) e^{(i\lambda -\mu)t} \cos(ts)\,dt \\
  &= \frac{1}{2} \mathcal F \left(  \beta_k(\cdot) e^{(i\lambda-\mu)\cdot} + \beta_k(-\cdot)
  e^{-(i\lambda-\mu)\cdot} \right)(s)
  \end{aligned} 
\end{align}
where $\mathcal F$ denotes the one-dimensional Fourier transform. 
For any given $r>2$ we choose $M\in\N$ such that $\sigma_{r'}\leq 2M$. We then obtain for all
$k\in\N$
 \begin{align*}
  &\|S_k\ast f\|_{L^2(\Hy^N)}^2 \\
  &= \Big\|\int_\R \beta_k(t) e^{(i\lambda -\mu)t} \cos(tP)\,dt  f\Big\|_{L^2(\Hy^N)}^2\\
  &=  \|\psi_k(P)  f\|_{L^2(\Hy^N)}^2\\
  &= \int_\R |\psi_k(s)|^2 \,d\skp{E_P(s)f}{f} \\
  &\leq  \int_\R |\psi_k(s)|^2 \|\dfrac{d}{ds} E_P(s)f\|_{L^r(\Hy^N)} \|f\|_{L^{r'}(\Hy^N)}  \,ds\\
  &\leq C \int_0^\infty |\psi_k(s)|^2 \left(s^2 1_{[0,1]}(s) + s^{\sigma_{r'}}1_{[1,\infty)}(s)\right)
  \|f\|_{L^{r'}(\Hy^N)}^2 \,ds  \\
  &\leq C \|f\|_{L^{r'}(\Hy^N)}^2 \cdot \int_0^\infty |\psi_k(s)|^2 (s^2+\ldots+s^{2M}) \,ds   
  \\
  &= C \|f\|_{L^{r'}(\Hy^N)}^2 \cdot  \int_\R   |(\mathcal F^{-1}\psi_k)'(s)|^2 + \ldots 
  + |(\mathcal F^{-1}\psi_k)^{(M)}(s)|^2\,ds  \\
   &\stackrel{\eqref{eq:defn_psik}}{=} 
    C \|f\|_{L^{r'}(\Hy^N)}^2 \cdot  \int_0^\infty 
   \Big|\frac{d}{ds}\left(\beta(2^{-k}s) e^{(i\lambda-\mu)s}\right) + \ldots
   + \frac{d^M}{ds^M}\left(\beta(2^{-k}s) e^{(i\lambda-\mu)s}\right)  \Big|^2 \,ds  \\
   &\leq  C\|f\|_{L^{r'}(\Hy^N)}^2 \cdot   \int_0^\infty \left(\beta(2^{-k}s)^2+\ldots+\beta^{(M)}(2^{-k}s)^2
   \right) e^{- 2\mu s} \,ds  \\
   &\leq C 2^k \|\beta\|_{H^M(\R)}^2 \|f\|_{L^{r'}(\Hy^N)}^2.  
\end{align*}
Taking the square root of this estimate we obtain by duality (note that $f\mapsto S_k\ast f$ is symmetric)
\begin{equation} \label{HSe1}
  \|S_k \ast f\|_{L^r (\Hy^N )} \leq C 2^{k/2}\|f\|_{L^2(\Hy^N)}.
\end{equation}
(This is an improved version of (4.13) in~\cite{HS}.) 

\medskip

Interpolating now \eqref{HSe1} and \eqref{HSe2} we get
$$
  \|S_k \ast f\|_{L^q(\Hy^N)}\leq C 2^{k(1-\frac{1}{p} + M(1-\frac{2}{p}))} \|f\|_{L^p(\Hy^N)}
  \qquad q=\frac{rp}{2(p-1)},r>2.
$$
So for any given $p\in (1,2),q>\frac{p}{p-1}$ we can choose $r>2$ as in the previous line and take $M$
sufficiently large to get 
$$
  \|S_k \ast f \|_{L^q(\Hy^N)}\leq C 2^{-k}\|f\|_{L^p(\Hy^N)},\quad\text{provided } 
  p\in (1,2),\; q>\frac{p}{p-1}>2.
$$
The dual version of this is
$$
  \|S_k \ast f \|_{L^q (\Hy^N)}\leq C 2^{-k}\|f\|_{L^p(\Hy^N)},\quad\text{provided }
  2<q<\frac{p}{p-1}<\infty
$$
and interpolating both yields  
$$
  \|S_k \ast f \|_{L^q(\Hy^N)}\leq C 2^{-k}\|f\|_{L^p(\Hy^N)},\quad\text{provided }
  p\in (1,2),\; q>2. 
$$
 The remaining case $p=1,q>2$ may again be obtained by interpolation with \eqref{HSe2} and we are done.
\end{proof}
  
   Having established locally uniform bounds and taking into account
   Theorem~I.4.2~\cite{IsoKur_Introduction} we may define
  \begin{align}\label{eq:defnRE}
    \mathcal R_\lambda + i \mathcal E_\lambda
    := (L-\lambda^2-i0)^{-1} 
    := \lim_{\mu\to 0^+} (L- (\lambda_\mu+i\mu)^2)^{-1} 
  \end{align}
  where $\lambda_\mu := \sqrt{\lambda^2+\mu^2}$ as bounded linear operators on Lebegue spaces in $\Hy^N$. The
  main properties of these operators are summarized in the following result.

 \begin{cor}\label{cor:LAP}
   The operators $\mathcal R_\lambda,\mathcal E_\lambda:L^p(\Hy^N)\to L^q(\Hy^N)$ are linear and
   bounded provided that $1\leq p<2<q$ and $\frac{1}{p}-\frac{1}{q}\leq \frac{2}{N}$ with $(p,q)\neq
   (1,\frac{N}{N-2}),(\frac{N}{2},\infty)$. Moreover, we have the representation formula 
   $\mathcal R_\lambda f= G\ast  f$
   for all $f\in C_0^\infty(\Hy^N)$ where the Green's function $G(t):=\lim_{\mu\to 0^+}
   \Real(G_{\lambda+i\mu}(t))$ satisfies
   \begin{align} \label{eq:estimate_G}
     \begin{aligned}
	  |G(t)| &\leq C\max\{|t|^{2-N},|\log(t)|\} &&(|t|\leq 1),\\ 
	  |G(t)| &\leq C e^{(1-N)t/2} &&(|t|\geq 1)
	  \end{aligned}
   \end{align}
    For all $f\in L^p(\Hy^N)$ the function $\mathcal R_\lambda f$ is a strong solution of
    $L\phi-\lambda^2\phi=f$ in~$\Hy^N$ and $\mathcal E_\lambda f$ solves $L\psi-\lambda^2\psi=0$ in~$\Hy^N$.
    Finally, we have the identities
\begin{align} \label{eq:identityRlambda}
  \begin{aligned}
  \int_{\Hy^N} f (\mathcal E_\lambda g)\,dV 
  =  \frac{\pi}{2\lambda}  \skp{f}{A_\lambda g},
  \qquad
  \int_{\Hy^N} f (\mathcal R_\lambda g)\,dV  
  = \text{p.v.} \int_\R \frac{\skp{f}{A_s g}}{s^2-\lambda^2} \,ds
  \end{aligned} 
\end{align}
  for all $f,g\in L^p(\Hy^N)$ with $1\leq p<2$ where $A_\lambda:L^p(\Hy^N)\to L^{p'}(\Hy^N)$ is given by
  \begin{equation}\label{eq:formula_spectraldensity}
    A_\lambda
    := \frac{d}{d\lambda} E_P(\lambda) 
    =   (\mathcal F_0(\lambda)^{(+)})^*\mathcal F_0(\lambda)^{(+)}
    =  (\mathcal F_0(\lambda)^{(-)})^*\mathcal F_0(\lambda)^{(-)}
  \end{equation}
  for the bounded linear operators $\mathcal F_0^{(\pm)}(\lambda):L^p(\Hy^N)\to L^2(\R^{N-1})$ and 
  $\mathcal F_0^{(\pm)}(\lambda)^*:L^2(\R^{N-1})\to L^{p'}(\Hy^N)$ defined in
  (I.4.2),(I.4.10)~\cite{IsoKur_Introduction}.
 \end{cor}
\begin{proof} 
  The asymptotics of $G$ follows from Proposition~\ref{prop:GreenFunctionH}.
  Since the boundedness of  $\mathcal R_\lambda,\mathcal E_\lambda$ result from
  Theorem~\ref{thm:Resolvent_estimates}, we next prove~\eqref{eq:identityRlambda} for $f,g\in
  C_0^\infty(\Hy^N)$. Let $h(s):= \skp{f}{A_sg}$. Then we have
 \begin{align*}
 \int_{\Hy^N} f (L-\lambda^2-i0)^{-1} g\,dV 
  &= \lim_{\mu\to 0^+}\int_{\Hy^N} f (L-(\lambda_\mu+i\mu)^2)^{-1} g\,dV \\
  &\stackrel{\eqref{eq:formula_resolvent}}{=} \lim_{\mu\to 0^+} \bigskp{f}{
  -\frac{1}{i\lambda_\mu-\mu}\int_0^\infty e^{(i\lambda_\mu-\mu)t} \cos(tP)g\,dt}
  \\
  &= - \frac{1}{i\lambda}  \lim_{\mu\to 0^+}  \bigskp{f}{\int_0^\infty  e^{(i\lambda_\mu-\mu)t}
  \Big(\int_\R \cos(ts)\,dE_P(s)g\Big)\,dt} \\
  &= - \frac{1}{i\lambda}  \lim_{\mu\to 0^+}  \int_\R \Big(\int_0^\infty e^{(i\lambda_\mu-\mu)t}
  \cos(ts) \,dt \Big) \,d\skp{f}{E_P(s)g} \\
  &= - \frac{1}{i\lambda}\lim_{\mu\to 0^+} \int_\R \frac{i(\lambda_\mu+i\mu)}{(\lambda_\mu+i\mu)^2-s^2}
  \,d\skp{f}{E_P(s)g} \\
  &= -  \lim_{\mu\to 0^+} \int_\R \frac{1}{\lambda^2-s^2+2i\mu\lambda_\mu}  \,d\skp{f}{E_P(s)g}  \\
  &= -  \lim_{\mu\to 0^+} \int_\R  \frac{h(s)}{\lambda^2-s^2+2i\mu\lambda_\mu}
  \,ds \\
  &=  \frac{i\pi}{2\lambda}  h(\lambda) - \lim_{\mu\to 0^+} \int_\R
  \frac{h(s)-h(\lambda)}{\lambda^2-s^2+2i\mu\lambda_\mu} \,ds \\
  &=  \frac{i\pi}{2\lambda}  h(\lambda) + \text{p.v.} \int_\R \frac{h(s)}{s^2-\lambda^2} \,ds. 
\end{align*}
 This computation and the definition~\eqref{eq:defnRE} yield~\eqref{eq:identityRlambda} by density of the
 test functions.  Moreover, for all $f\in C_0^\infty(\Hy^N)$, we get from \eqref{eq:identityRlambda} and 
 (I.4.3)~\cite{IsoKur_Introduction} the identity
 \begin{align*}
   \skp{f}{A_\lambda f} 
   = \frac{\lambda}{\pi i} \bigskp{\left( (L-\lambda^2-i0)^{-1} - (L-\lambda^2+i0)^{-1} \right)f}{f}
   = \|\mathcal F_0^{(\pm)}(\lambda)f\|_{L^2(\R^{N-1})}^2
 \end{align*}
 for the operators $\mathcal F_0^{(\pm)}(\lambda)$ defined in (I.4.2)~\cite{IsoKur_Introduction}. Hence, since
 $A_\lambda$ is a symmetric operator, we deduce
 $$
   \skp{f}{A_\lambda g} 
   = \skp{ \mathcal F_0^{(\pm)}(\lambda)f}{\mathcal F_0^{(\pm)}(\lambda)g}_{L^2(\R^{N-1})}
 $$
 for all test functions $f,g$ and~\eqref{eq:formula_spectraldensity} follows. Finally, since $\mathcal
 E_\lambda:L^p(\Hy^N)\to L^{p'}(\Hy^N)$ is bounded, the operators $\mathcal F_0^{(\pm)}(\lambda)^*$ and
 $\mathcal F_0^{(\pm)}(\lambda)$ are bounded as well.
\end{proof}
 
 The functions $\mathcal F_0(\lambda)^*g$ with $g\in L^2(\R^{N-1})$ actually represent the totality of
 solutions of the homogeneous Helmholtz equation by Theorem~I.4.3~\cite{IsoKur_Introduction}. They 
 are the counterparts of Euclidean Hergoltz waves that are defined as the images of
 $L^2(S_\lambda)$-densities under the adjoint of the Fourier restriction operator $f\mapsto \hat
 f|_{S_\lambda}$, where $S_\lambda\subset\R^N$ denotes the sphere of radius $\lambda$.
 While hyperbolic Herglotz waves $\mathcal F_0(\lambda)^*g$ lie in $L^p(\Hy^N)$ for all $p>2$, the optimal
 $L^p$ decay rate of Euclidean Herglotz waves is given by the Stein-Tomas Theorem saying that $p\geq
 \frac{2(N+1)}{N-1}$. Better decay properties of the latter, namely pointwise decay like $|x|^{(1-N)/2}$ at
 infinity, can be obtained for densities of higher regularity via the method of
 stationary phase. Whether $L^p(\R^N)$ can be reached in the optimal range $p>\frac{2N}{N-1}$ only by
 assuming stronger integrability assumptions on the density, is a delicate question related to the
 Restriction Conjecture which  is still unsolved for $N\geq 3$. So we see that hyperbolic Herglotz waves
 have much better integrability properties than their Euclidean counterparts. This will allow us to adopt a 
 fixed point approach that is decidedly simpler than its Euclidean analogue~\cite{Gu,Man_uncountably} where
 Helmholtz equations of the form~\eqref{eq:NLH} can only be discussed for a restricted set of exponents $p$.

\section{Proof of Theorem~\ref{thm:NoStrichartz}} \label{sec:Nostrichartz}
 
Given the results of the previous section, the proof is quite simple. Let $g\in L^2(\R^{N-1})$ be nontrivial.
Then Corollary~\ref{cor:LAP} implies $\psi_0:=\mathcal F_0^{(+)}(\lambda)^*g \in L^r(\Hy^N)$ for all $r>2$
and $\psi(x,t):= e^{i\omega t}\psi_0(x)$ solves~\eqref{eq:Strichartzhyp}, but $\psi\notin
L^p(\R,L^q(\Hy^N))$, since $\psi$ is periodic in time. This proves the result. \qed

\section{Proof of Theorem~\ref{thm:existence_small_solutions}} \label{sec:SmallSolutions}
 
 In this section we prove Theorem~\ref{thm:existence_small_solutions} with the aid of the Contraction Mapping
 Principle and the estimates for the operators $\mathcal
 F_0(\lambda)^{(+)},\mathcal R_\lambda$ from Corollary~\ref{cor:LAP}. As in
 \cite{Man_uncountably} we use a smooth function $\chi\in C^\infty(\R)$ such that $\chi(z)=z$ for $|z|\leq
 \frac{1}{2}$, $\chi(z)=1$ for $|z|\geq 1$ and define, for any given $g\in L^2(\R^{N-1})$, the operator 
$$
  T_g(u):= \mathcal F_0^{(+)}(\lambda)^* g + \mathcal R_\lambda(f(\cdot,\chi(u))). 
$$ 
As mentioned in the introduction we use the assumption $|f(x,z)|+|z||f_z(x,z)|\leq C|z|^{q-1}$. The operator
$T_g$ is well-defined as a map from $L^s(\Hy^N)$ to $L^s(\Hy^N)$ provided we choose $s$ according to
$\max\{2,\frac{N(q-2-\delta)}{2},q-1-\delta\}<s<2(q-1-\delta)$ for some $\delta\in (0,q-2)$.
Indeed, under this assumption Corollary~\ref{cor:LAP} applies and we obtain 
\begin{align*}
  \|T_g(u)\|_{L^s(\Hy^N)} 
  &\leq \|\mathcal F_0^{(+)}(\lambda)^* g\|_{L^s(\Hy^N)} + \|\mathcal R_\lambda(f(\cdot,\chi(u)))\|_{L^s(\Hy^N)} \\
  &\leq C(\|g\|_{L^2(\R^{N-1})} + \|f(\cdot,\chi(u))\|_{L^{\frac{s}{q-1-\delta}}(\Hy^N)}) \\
  &\leq C(\|g\|_{L^2(\R^{N-1})} + \||\chi(u)|^{q-1}\|_{L^{\frac{s}{q-1-\delta}}(\Hy^N)}) \\
  &\leq C(\|g\|_{L^2(\R^{N-1})} + \||\chi(u)|^{q-1-\delta}\|_{L^{\frac{s}{q-1-\delta}}(\Hy^N)}) \\
  &\leq C(\|g\|_{L^2(\R^{N-1})} + \|u\|_{L^s(\Hy^N)}^{q-1-\delta}).  
\end{align*}
So we conclude that $T_g$ is a selfmap on any given sufficiently small ball in $L^s(\Hy^N)$ provided $g\in
L^2(\R^{N-1})$ is chosen small enough. Similarly, using $|f_z(x,z)|\leq C|z|^{q-2}$ for $|z|\leq 1$ we obtain
that $T_g$ is a contraction on small balls. Hence, by the Contraction Mapping Principle, for every given small
enough $g\in L^2(\R^{N-1})$ the operator $T_g$ has a unique fixed point in a small ball and thus a solution $u\in L^s(\Hy^N)$ 
of $Lu-\lambda^2 u = f(x,\chi(u))$. Elliptic $L^p$-estimates imply $u\in L^\infty(\Hy^N)\cap L^s(\Hy^N)$ and
using the mapping properties of $ \mathcal F_0^{(+)}(\lambda)^*,\mathcal R_\lambda$ from
Corollary~\ref{cor:LAP} iteratively, we actually find $u\in L^r(\Hy^N)$
for all $r\in (2,\infty]$. Applying global $L^p$-estimates from 
from Theorem~A~\cite{Tay_Lpestimates} we find $u\in W^{2,r}(\Hy^N)$ for all $r\in
(2,\infty)$. Moreover, choosing $g$ sufficiently small, we may choose the ball and hence the
$L^s(\Hy^N)$-norm of $u$ so small that $\|u\|_{L^\infty(\Hy^N)}\leq C\|u\|_{L^s(\Hy^N)} \leq \frac{1}{2}$
holds.
But then we have $\chi(u)=u$ and $u$ is the solution of the nonlinear Helmholtz equation~\eqref{eq:NLH}
we were looking for. We finally mention that different $g$ yield different solutions since $\mathcal
F_0^{(+)}(\lambda)^*$ is injective, see Corollary~I.4.6~\cite{IsoKur_Introduction}. \qed

\section{Proof of Theorem~\ref{thm:existence_largesolutions}}  \label{sec:LargeSolutions}

In this section we prove the existence of nontrivial solutions to 
$$
   Lu -\lambda^2  u = \Gamma |u|^{p-2}u\quad\text{in }\Hy^N
$$
under the assumptions of Theorem~\ref{thm:existence_largesolutions}
using dual variational methods. In the context of  the nonlinear Helmholtz equation, this approach was
introduced  by Evequoz and Weth  in order to treat the corresponding equation in $\R^N$ for
$N\geq 3$ \cite{EW} or $N=2$ \cite{Ev}. We show that many of their ideas carry over to the Euclidean setting.
It actually turns out that the main difficulty in their approach, namely the verification of the ''Nonvanishing Property'', is much simpler due to a variant
of the Stein-Kunze phenomenon, as we will show later. Given that the fundamental ideas are all present in the
literature, we keep the presentation short. Following~\cite{EW} the dual variational method works as follows.
Using that $\Gamma$ is nonnegative, we may set $v:=\Gamma^{1/p^\prime} |u|^{p-2}u\in L^{p^\prime}(\Hy^N)$ so
that the task is to find nontrivial solutions of \eqref{eq:NLH} by solving the integral equation
\begin{equation} \label{eq:dual_equation}
 \Gamma^{1/p} \mathcal R_\lambda (\Gamma^{1/p}v) = |v|^{p^\prime -2}v. 
\end{equation} 
Notice that the mapping properties of $\mathcal R_\lambda$ from Corollary~\ref{cor:LAP} ensure that this
equation makes sense for $v\in L^{p'}(\Hy^N)$ as long as $2<p<2^*$. Since  $\mathcal R_\lambda$ is
symmetric, solutions of \eqref{eq:dual_equation} are critical points of the functional $J\in
C^1(L^{p'}(\Hy^N),\R)$ defined by
\begin{equation} \label{eq:def_J}
  J(v) := \frac{1}{p^\prime} \int_{\Hy^N}  |v|^{p^\prime} \,dV 
  - \frac{1}{2} \int_{\Hy^N}   \Gamma^{1/p} v \mathcal R_\lambda( \Gamma^{1/p}  v)  \,dV.
\end{equation} 
In the proof of our statements (i),(ii),(iii), we will apply Critical Point Theory to prove the existence of
one respectively infinitely many nontrivial critical points of $J$. For the same reasons as in the
previous section, these solutions are actually strong solutions and belong to $W^{2,r}(\Hy^N)$ for all
$r\in (2,\infty)$. 

\medskip

\noindent\textbf{\textit{Proof of~(i):}} We only prove the claim for constant $\Gamma$.
Note that the proof in this special case requires the verification of the ''Nonvanishing property''
from~\cite{EW} so that the claim in the general case $\Gamma\geq \Gamma_0=\lim_{x\to\infty} \Gamma(x)>0$
follows precisely as in Theorem~4.3~\cite{Ev:OnThe}. So from now on we assume w.l.o.g. $\Gamma\equiv 1$ so
that $J(u)=J(u\circ\tau)$ for all hyperbolic tranlations $\tau$.

\medskip

As in Lemma~4.2~(i) \cite{EW} 
one finds that $J$ has the mountain pass geometry and that there is a Palais-Smale sequence $(v_n)$ in
$L^{p'}(\Hy^N)$ at its Mountain Pass level $c>0$ given by 
\begin{equation}\label{eq:MP_level}
  c = \inf_{\gamma \in P} \max_{t\in [0,1]}J(\gamma (t)) >0.
\end{equation}
In other words,
\begin{align}  \label{eq:PScondition}
\begin{aligned}
  |v_n|^{p'-2}v_n - G\ast v_n \to 0 \quad\text{in }L^p(\Hy^N),\\
  \frac{1}{p'} \int_{\Hy^N} |v_n|^{p'}\,dV - \frac{1}{2} \int_{\Hy^N} v_n (G\ast v_n)\,dV \to c.
  \end{aligned}
\end{align}
From this one infers that $(v_n)$ is bounded in $L^{p'}(\Hy^N)$ and 
\begin{equation}\label{eq:nonvanishing0}
  \left(\frac{1}{p'}-\frac{1}{2}\right) \int_{\Hy^N} v_n (G\ast v_n)\,dV \to c.
\end{equation}
Next we show that after some hyperbolic translations $(v_n)$ converges to some
nontrivial critical point $v\in L^{p'}(\Hy^N)$ of $J$. 

\medskip

The Stein-Kunze estimate from Lemma~4.1~\cite{AnkPi_WaveKG} yields  
\begin{align*}
  &\left(\frac{1}{p'}-\frac{1}{2}\right) \left| \int_{\Hy^N} v_n \left[ ( 1_{[\frac{1}{R},R]^c} G)\ast v_n
  \right] dV \right|
  \\
  & \leq \|v_n\|_{L^{p'}(\Hy^N)} \|(  1_{[\frac{1}{R},R]^c} G) \ast v_n\|_{L^p (\Hy^N)}  \\
  &\leq  C \|v_n\|_{L^{p'}(\Hy^N)}^2 \left(\int_0^{1/R} + \int_R^\infty 
  (\sinh r)^{N-1} (1+r) e^{-(N-1)r/2} |G(r)|^{p/2}\,dr\right)^{2/p}  \\
  &\leq \frac{c}{2}
\end{align*}
for all  $n\in\N$ provided $R$ is large enough. Here we used the boundedness of $(v_n)$ in
$L^{p'}(\Hy^N)$, the asymptotics of $G$ from~\eqref{eq:estimate_G} and $2<p<2^*$. So we have
$$
  \left(\frac{1}{p'}-\frac{1}{2}\right) \liminf_{n\in\N} \int_{\Hy^N} v_n (1_{[\frac{1}{R},R]}G\ast v_n)\,dV
  \geq \frac{c}{2}.
$$
 Next let $(Q_l)_{l\in\N}$ be a family of disjoint geodesic balls of radius $R$ 
 the centers of which have the geodesic distance $R/2$ and that
 cover the whole $\Hy^N$. Denoting by $2Q_l$ the ball with the same
 center but doubled radius we get for almost all $n$ 
\begin{align*}
  \frac{c}{4}
  &\leq \left(\frac{1}{p'}-\frac{1}{2}\right)   \int_{\Hy^N} v_n  \left[(1_{[\frac{1}{R},R]} G)\ast v_n
  \right] \,dV    \\
  & \leq C\sum_{l=1}^\infty
  \int_{Q_l} \left(\int_{1/R < d(x,y)<R} |G(d(x,y))||v_n (x)||v_n (y)| \,dV(y)\right)\,dV(x) \\
  &\leq C  \max_{1/R \leq d(x,y)\leq R} |G(d(x,y))|  \sum_{l=1}^\infty \int_{Q_l} \left(\int_{2Q_l}
  |v_n(x)||v_n(y)|\,dV(y) \right)\,dV(x) \\
   &\leq  C   \sum_{l=1}^\infty  \left(\int_{2Q_l} |v_n (x)|\,dV(x)\right)^2 \\
   &\leq  C \sum_{l=1}^\infty  \left(\int_{2Q_l} |v_n (x)|^{p'}\,dV(x)\right)^{2/p'} \\
   &\leq  C  
    \left( \sup_{m\in\N} \int_{2Q_m} |v_n (x)|^{p'}\,dV(x)\right)^{2/p'-1}  \cdot 
   \sum_{l=1}^\infty  \int_{2Q_l} |v_n (x)|^{p'}\,dV(x)  \\
   &\leq  C  
    \left( \sup_{m\in\N} \int_{2Q_m} |v_n (x)|^{p'}\,dV(x)\right)^{2/p'-1}  \cdot 
    \|v_n\|_{L^{p'}(\Hy^N)} \\
   &\leq  C  \left( \sup_{m\in\N} \int_{2Q_m} |v_n (x)|^{p'}\,dV(x)\right)^{2/p'-1}.    
\end{align*}
Here we used that the balls $2Q_m$ cover $\Hy^N$ only a finite number of times because $\Hy^N$ has bounded
geometry. 
The latter estimate implies that there are centers $x_n\in\Hy^N$ such that 
\begin{equation}\label{eq:nonvanishing}
  \liminf_{n\to\infty} \int_{B_{2R}(x_n)} |v_n|^{p'}\,dV  >0.
\end{equation}
Denoting by $\tau_{x_n}$ the hyperbolic translation with $\tau_{x_n} 0=x_n$, 
we obtain that $w_n(x):= v_n (\tau_{x_n} x)$ is another Palais-Smale sequence of $J$. Given that it is bounded
as well, it converges weakly to some $w\in L^{p'}(\Hy^N)$. Combining the first line of \eqref{eq:PScondition}
with  local $L^p$-estimates, we infer that $w_n$ is bounded in $W^{2,p}(B_{4R}(0))$ and hences converges in
$L^{p'}(B_{2R}(0))$ to its weak limit $w$ so that \eqref{eq:nonvanishing} implies $w\neq 0$. Moreover, one
checks that $w$ is a critical point of $J$ at the mountain pass level and the proof is finished. 
\qed

\medskip

\noindent\textbf{\textit{Proof of~(ii):}} Using the formula~\eqref{eq:identityRlambda} we may verify the
assumptions of the Symmetric Mountain Pass Theorem as in Lemma~3.2~\cite{MMP}. Indeed, for
every $m\in\N$, we can choose radially symmetric functions $\psi_1,\ldots,\psi_m\in C_0^\infty(\R_+)$ with
mutually disjoint supports contained in the exterior of the ball of radius $\lambda$. Then
$\{z_1,\ldots,z_m\}$ is a linearly independent set if we set 
$$
  z_j:= \max\{\Gamma^{-1/p},\delta\}\cdot \psi_j(P)h
$$ 
for some fixed $h\in C_0^\infty(\Hy^N)$ and sufficiently small $\delta>0$. Indeed, due
to~\eqref{eq:identityRlambda} these functions are even mutually orthogonal and we have
$J(tz_j)\to -\infty$ as $t\to\infty$ because of 
\begin{align*}
  J(tz_j)
  &= \frac{t^{p'}}{p'} \|z_j\|_{L^{p'}(\Hy^N)}^{p'} - \frac{t^2}{2} \int_{\Hy^N} \Gamma^{1/p}z_j \mathcal
  R_\lambda(\Gamma^{1/p}z_j)  \,dV \\
   &\leq\frac{t^{p'}}{p'}\|z_j\|_{L^{p'}(\Hy^N)}^{p'} - \frac{t^2}{4} \int_{\Hy^N} \psi_j(P)h \mathcal
   R_\lambda(\psi_j(P)h) \,dV \\  
  &\stackrel{\eqref{eq:identityRlambda}}{=} \frac{t^{p'}}{p'} \|z_j\|_{L^{p'}(\Hy^N)}^{p'} - \frac{t^2}{4}
   \text{p.v.} \int_\R \frac{\skp{\psi_j(P)h}{\frac{d}{ds}E_P(s)(\psi_j(P)h)}} {s^2-\lambda^2}\,ds \\
  &= \frac{t^{p'}}{p'} \|z_j\|_{L^{p'}(\Hy^N)}^{p'} - \frac{t^2}{4}
  \text{p.v.} \int_\R \frac{|\psi_j(s)|^2 }{s^2-\lambda^2} \skp{h}{\frac{d}{ds}E_P(s)h}\,ds  \\
   &= \frac{t^{p'}}{p'} \|z_j\|_{L^{p'}(\Hy^N)}^{p'} - \frac{t^2}{4}
   \int_{\supp(\psi_j)}  \frac{\psi_j(s)^2}{s^2-\lambda^2} \|\mathcal F_0^{(+)}(s)h\|_{L^2(\R^{N-1})}^2 \,ds, 
\end{align*}
which tends to $-\infty$ if $h$ is chosen suitably.
From $\Gamma(x)\to 0$ as $x\to\infty$ one deduces as in Lemma~5.2~\cite{EW} that the Palais-Smale
condition holds so that the existence of an unbounded sequence of solutions follows from the Symmetric
Mountain Pass Theorem, see Theorem~6.5~\cite{Str_Variational} for the Euclidean version.  \qed

\medskip

\noindent\textbf{\textit{Proof of (iii):}} We show that under the assumptions of part (iii) the
functional $J$ restricted to $L^{p'}_{rad}(\Hy^N)$ satisfies again the assumptions of the Symmetric Mountain
Pass Theorem. First, the functions $z_1,\ldots,z_m$ from above are radial if $h$ is radial, so it remains to verify
the Palais-Smale condition. A Palais-Smale sequence $(v_n)$ in $L^{p'}_{rad}(\Hy^N)$ is bounded and hence without loss of generality weakly
convergent to some $v\in L^{p'}_{rad}(\Hy^N)$. Since $L^{p'}_{rad}(\Hy^N)$ is a uniformly convex Banach space, the convergence
$v_n\to v$ in $L^{p'}_{rad}(\Hy^N)$ is proved once we show $\|v_n\|_{L^{p'}(\Hy^N)} \to
\|v\|_{L^{p'}(\Hy^N)}$, see Proposition 3.32~\cite{Bre_FuncAna}. In view of~\eqref{eq:PScondition} this holds
once we show that  $\psi_n:=G\ast v_n$ has a convergent subsequence in $L^p_{rad}(\Hy^N)$. This is checked as
follows. Corollary~\ref{cor:LAP} implies
that $\psi_n:= G\ast v_n = \mathcal R_\lambda v_n$ is bounded in $L^q_{rad}(\Hy^N)$ for all $2<q<
\frac{Np'}{(N-2p')_+}$ because $v_n$ is bounded in $L^{p'}_{rad}(\Hy^N)$. Then 
$L\psi_n + k^2 \psi_n = (k^2+\lambda)\psi_n +  v_n \in L^q_{rad}(\Hy^N)+L^{p'}_{rad}(\Hy^N)$ 
and the $L^p$-estimates from Theorem~A~\cite{Tay_Lpestimates}
show that $(\psi_n)$ is bounded in $W^{2,q}_{rad}(\Hy^N)+W^{2,p'}_{rad}(\Hy^N)$. By Theorem~2
in~\cite{HebVau_Sobolev} (see also Theorem~3.1~\cite{BhaSan_Poincare} for a more elementary proof of a related result), 
this space imbeds compactly into $L^p_{rad}(\Hy^N)$ if $q$ is chosen smaller than but sufficiently close to
$p\in (2,2^*)$.
So $(\psi_n)$ has a convergent subsequence in $L^p_{rad}(\Hy^N)$ and $J$ restricted to $L^{p'}_{rad}(\Hy^N)$
satisfies the Palais-Smale condition. As above, we obtain an unbounded sequence of critical points of $J$,
which finishes the proof. \qed

\section{Proof of Theorem~\ref{thm:radial}}   \label{sec:RadialSolutions}

For the proof we have to analyze the unique solution of the ODE initial value problem
\begin{equation}\label{eq:ODE}
  - u'' - \frac{f'(r)}{f(r)}u' - V(r) u = \Gamma(r)|u|^{p-2}u,\qquad  u(0)=\gamma,\; u'(0)=0
\end{equation}
where $\gamma$ will be assumed to be positive without loss of generality. The first step is to find suitable
bounds for $u,u',u''$. To this end we introduce the positive function
 \begin{equation} \label{eq:defn_Z}
  Z(r) := \frac{1}{2} u'(r)^2 + \frac{1}{2}V(r)u(r)^2 + \frac{1}{p}  \Gamma(r)|u(r)|^p.
 \end{equation}
 From $f'\geq 0$ and (H2),(H3) we get that there is an integrable function $m$ such that  
\begin{align} \label{eq_Zprime}
  \begin{aligned}
  Z'(r)
  &\stackrel{\eqref{eq:defn_Z}}{=} u'(r) \left( u''(r)+V(r)u(r)+\Gamma(r)|u(r)|^{p-2}u(r)\right) \\
  &\quad + \frac{1}{2} V'(r)u(r)^2 +
  \frac{1}{p}\Gamma'(r)|u(r)|^p  \\
  &\stackrel{\eqref{eq:ODE}}{=} - \frac{f'(r)}{f(r)}|u'(r)|^2
    + \frac{1}{2} V'(r)u(r)^2 + \frac{1}{p}\Gamma'(r)|u(r)|^p \\
  &\leq m(r) Z(r)  
  \end{aligned}
\end{align}
and thus
\begin{equation} \label{eq:estimateZ_0R}
  Z(r) 
  \leq Z(0) \exp\left(\int_0^\infty m(s)\,ds\right) 
  \qquad \text{for all }r>0.
\end{equation}
 Since $\Gamma$ is nonnegative and $V$ is positive we deduce that $u,u',u''$ exist globally.

\medskip

Next we show that $u$ has an unbounded sequence of zeros. Indeed, the function $v(r):=f(r)^{1/2} u(r)$
satisfies
\begin{align}\label{eqv}
  \begin{aligned}
  v^{\prime \prime}(r)&+c(r)v(r)=0,\qquad\text{where } \\ 
  c(r) &:= \Gamma(r)|u(r)|^{p-2}+ V(r) -  \frac{f''(r)}{2f(r)} + \frac{f'(r)^2}{4f(r)^2}  \\
   &= \Gamma(r)|u(r)|^{p-2}+ V(r) -  \dfrac{1}{2} \log(f)''(r) - \dfrac{1}{4}   (\log(f)'(r))^2 
  \end{aligned}
\end{align}
From $\Gamma\geq 0$, (H1), (H2) we deduce 
$$
  \liminf_{r\to\infty} c(r) \geq V_\infty - \frac{\kappa^2}{4}>0.
$$
Hence, the function $c$ is uniformly positive near infinity so that Sturm's oscillation theorem  
implies that $v$ and hence $u$ has an unbounded sequence of zeros.

 \medskip
 
 Next we prove the estimates~\eqref{notL2}. To this end we define
\begin{equation}\label{def:psi}
  \psi(r):= \frac{1}{2} v'(r)^2+ f(r)\left(\frac{1}{2} \tilde V(r)
  u(r)^2+\frac{1}{p}\Gamma(r)|u(r)|^p\right),
\end{equation}
where  $\tilde V(r):=V(r)-\frac{\kappa^2}{4}$. Differentiation yields  
\begin{align*} 
\psi'
&= v'v'' + f'\left(\frac{1}{2} \tilde V u^2+\frac{1}{p}\Gamma|u|^p \right) \\
&+ f\left( \frac{1}{2} \tilde V' u^2 +  \frac{1}{p}\Gamma'|u|^p  
   + ( \tilde V u+\Gamma|u|^{p-2}u)u' \right) \\
&\stackrel{\eqref{eqv}}{=} -cvv' +  f u' ( \tilde V u+\Gamma|u|^{p-2}u) \\
& + f'\left(\frac{1}{2} \tilde V u^2+\frac{1}{p}\Gamma|u|^p\right) 
 + f\left(  \frac{1}{2}   V' u^2 + \frac{1}{p}\Gamma'|u|^p \right) \\
&= -cvv' +  \left(f^{1/2} v' - \frac{1}{2}f'f^{-1/2}v\right)   ( \tilde V u+\Gamma|u|^{p-2}u ) \\
& + f'\left( \frac{1}{2} \tilde V u^2 +\frac{1}{p}\Gamma|u|^p\right) 
+ f\left( \frac{1}{2} V' u^2 + \frac{1}{p}\Gamma'|u|^p \right) \\
&= \left( -c + \tilde V +\Gamma|u|^{p-2} \right) vv' +
\big(\frac{1}{p}-\frac{1}{2}\big) f'\Gamma|u|^p + f\left( \frac{1}{2} V' u^2 + 
\frac{1}{p}\Gamma'|u|^p\right) \\
&= \left(  \frac{f''}{2f}  - \dfrac{(f')^2}{4f^2} - \frac{\kappa^2}{4}\right)  vv'  \\
& +\left( \big(\frac{1}{p}-\frac{1}{2}\big)\frac{f'}{f} \Gamma|u|^{p-2} + \frac{1}{2} V' +
\frac{1}{p}\Gamma'|u|^{p-2} \right)v^2.
\end{align*}
To prove the upper bounds in~\eqref{notL2} we prove an upper bound for $\psi$ as follows. From the previous
identity we get on the interval $[R,\infty)$ for large $R$
\begin{align*}
  \psi'
  &\leq \left|\frac{f''}{2f}  - \dfrac{(f')^2}{4f^2} - \frac{\kappa^2}{4}\right| \frac{v^2+(v')^2}{2} 
     + \left( \frac{|V'|}{\min_{[R,\infty)} \tilde V} + m\right) \left(  \frac{1}{2} \tilde V + \frac{1}{p}
     \Gamma|u|^{p-2} \right)v^2
     \\
  &\leq \frac{1}{\min\{1,\min_{[R,\infty)} \tilde V\}}  \left( \left|\frac{f''}{2f}  -
  \dfrac{(f')^2}{4f^2} - \frac{\kappa^2}{4}\right| + |V'|+ m\right)  \psi. 
\end{align*} 
Here we used $p>2,f'\geq 0,\Gamma\geq 0$ and $\min_{[R,\infty)} \tilde V>0$ for sufficiently large $R$ by
(H2). Since the prefactor is integrable over $[R,\infty)$ and $\psi$ is positive on this interval, we infer
from Gronwall's inequality that 
$$
  \psi(r)\leq C\psi(R) \qquad \text{for all }r\geq R
$$
where $C$ is independent of $\gamma$. Combining this inequality, the simple inequality
$$
  \psi(R) \leq 2\left(f(R)+\frac{f'(R)^2}{4f(R) \min_\R V}\right) Z(R)
$$ 
and \eqref{eq:estimateZ_0R} we get for some $A>0$ independent of $\gamma$ 
$$
  Z(r) + |\psi(r)|\leq A Z(0) = A(V(0)\gamma^2+\Gamma(0)|\gamma|^p) \qquad \text{for all }r\geq 0.
$$ 
This proves the upper estimate in~\eqref{notL2}. This upper estimate may be used in the proof of the lower
estimate. Indeed, $|u(r)|\leq C_\gamma (1+f(r))^{-1/2}$ implies
\begin{align*}
  \psi'
  &\geq - \left|\frac{f''}{2f}  - \dfrac{(f')^2}{4f^2} - \frac{\kappa^2}{4}\right| \frac{v^2+(v')^2}{2} 
     - \left( \|\Gamma\|_\infty C_\gamma^{p-2} f'f^{-p/2}  
     + \frac{|V'|}{\min_{[R,\infty)} \tilde V} + m\right)  \psi \\
  &\geq  - \frac{1}{\min\{1,\min_{[R,\infty)} \tilde V\}} \left( \left|\frac{f''}{2f}  -
  \dfrac{(f')^2}{4f^2} - \frac{\kappa^2}{4}\right| + 
   \|\Gamma\|_\infty C_\gamma^{p-2} f'f^{-p/2} + |V'| + m\right)  \psi. 
\end{align*}  
Again, the prefactor is integrable over $[R,\infty)$ and we obtain 
$$
  \psi(r)\geq c\psi(R) \qquad\text{for all }r\geq R
$$
where $c$ only depends on the $L^1$-norm of this prefactor. Since $u$ oscillates, we may choose $R$ such
that additionally $u(R)u'(R)=0$ holds. For such $R$ one has the simple inequality $\psi(R) \geq f(R)Z(R)$ so
that the differential inequality for $Z'$ finally yields a positive number $B>0$ independent of
$\gamma$ such that 
$$ 
  \psi(r)\geq B Z(0) = B(V(0)\gamma^2+\Gamma(0)|\gamma|^p) \qquad \text{for all }r\geq 0. 
$$ 
This finishes the proof of~\eqref{notL2}.    \qed

  \section*{Acknowledgements}
 The first author is supported by MIS F.4508.14 (FNRS), PDR T.1110.14F (FNRS). The second author acknowledges financial support by the Deutsche Forschungsgemeinschaft (DFG,
    German Research Foundation) through the Collaborative Research Center 1173.

\bibliographystyle{plain}
\bibliography{bibhelmholtz}
\end{document}